\documentclass[12pt]{amsart}

\usepackage{kantlipsum} 

\calclayout

\usepackage{amsmath, amssymb}
\usepackage{amscd}
\usepackage{verbatim}
\usepackage{tikz}
\usetikzlibrary{arrows,chains,matrix,positioning,scopes}

\makeatletter
\tikzset{join/.code=\tikzset{after node path={%
\ifx\tikzchainprevious\pgfutil@empty\else(\tikzchainprevious)%
edge[every join]#1(\tikzchaincurrent)\fi}}}
\makeatother

\tikzset{>=stealth',every on chain/.append style={join},
         every join/.style={->}}

\newtheorem{definition}{Definition}[section]
\newtheorem{theorem}[definition]{Theorem}

\newtheorem{proposition}[definition]{Proposition}
\theoremstyle{definition}
\newtheorem{remark}[definition]{Remark}

\newtheorem{example}[definition]{Example}

\newcommand{\noi}{\noindent}
\newcommand{\ra}{\rightarrow}

\begin{document}
	
	\title[Free Product on Semihypergroups]{Free Product on Semihypergroups}
	
	\author[C.~Bandyopadhyay]{Choiti Bandyopadhyay}

	\address{Department of Mathematical and Statistical Sciences, University of Alberta, Canada T6G 2G1}
	\email{choiti@ualberta.ca}

	\thanks{Part of this work is included in author's PhD thesis at the University of Alberta, Canada}

	\begin{abstract}
		In a previous paper \cite{CB}, we initiated a systematic study of semihypergroups and had a thorough discussion about some important analytic and algebraic objects associated to this class of objects. In this paper, we investigate free structures on the category of semihypergroups. We show that the natural free product structure along with the natural topology, although fails to give a free-product for topological groups, works well on a vast non-trivial class of `pure' semihypergroups containing most of the well-known examples including non-trivial coset and orbit spaces.
	\end{abstract}
	
	\maketitle

	\section{Introduction}
	\label{intro}
	
	\quad The concept of semihypergroups arises naturally in abstract harmonic analysis in terms of left/right coset spaces and orbit spaces of certain actions which appear frequently in different areas of research including Lie groups, coset theory, homogeneous spaces, dynamical systems, ordinary and partial differential equations and orthogonal polynomials, to name a few. These objects arising from locally compact groups, although retain some interesting structures from the parent category, fail to be a topological group, semigroup or even a hypergroup \cite{CB,JE}. Hence these kinds of objects that frequently appear while studying the classical theory of topological groups, fall out of the parent category and can not quite be studied or analysed with the existing general theory for hypergroups and topological groups.
	

	\quad The concept of a semihypergroup was first introduced around 1972 by C. Dunkl \cite{DU}, I. Jewett \cite{JE} and R. Spector \cite{SP}  independently, in pursuit of developing an unified theory involving convolution of measures that includes a wide array of objects arising from locally compact groups. In a sense, they serve as building blocks for hypergroup-axioms.

	
	\quad A semihypergroup, as one would expect, can be perceived as simply a generalization of a locally compact semigroup, where the product of two points is a certain probability measure, rather than a single point. But unlike hypergroups, it does not need to have an identity or involution (analogous to inverses in groups). Hence in a nutshell, a semihypergroup is essentially a locally compact topological space where the measure space is equipped with a certain convolution product, turning it into an associative algebra, which need not have an identity or an involution. It  can be seen (details in the following section) that all the orbit spaces and coset spaces can naturally be given a semihypergroup structure, though not a hypergroup structure. 
	
	
	\quad The fact that semihypergroups contain more generalized objects arising from different fields of research than classical group and hypergroup theory and yet sustains enough structure to allow an independent theory to develop, makes it an intriguing and useful area of study with essentially a broader area of applications. Unlike hypergroups, there is a severe lack of any extensive prior research on semihypergroups since its inception. On the other hand, almost all of the important concepts and results on hypergroups and locally compact groups fail to extend naturally to semihypergroups due to the lack of an involution, an identity, a standardized Haar measure and even an algebraic structure on the underlying space in semihypergroups. 

	
	\quad Hence in a previous paper \cite{CB} we initiated developing a systematic extensive theory on semihypergroups. Our approach towards developing a systematic theory was to define and explore some basic algebraic and analytic aspects and areas on semihypergroups, which are imperative for the development of any category of analytic objects in general. In the first installment of the study \cite{CB} we introduced and developed the concepts and theories of homomorphisms, ideals, kernels and almost periodic and weakly almost periodic functions on a semihypergroup. In this paper, we advance the theory by discussing free structures and universal properties on a class of semihypergroups.

	
	\quad  We first discuss how a semihypergroup homomorphism (introduced in \cite{CB}) naturally induces a homomorphism (in the classical sense) between measure algebras. Using this property of homomorphisms, we discuss how free products can be defined given a class of semihypergroups, abiding by the general norms of defining a free product given any category of objects in general. We then introduce a sub-class of `pure' non-trivial semihypergroups and investigate in detail the existence of a unique (upto isomorphism) free product given such a family of objects. 
	
	
	\quad The most interesting observation that comes out of this particular path of research is that the natural free product structure with the naturally induced topology, which fails to give a free product structure for topological groups \cite{MO}, indeed works well here and gives us a natural free product structure for a class of non-trivial semihypergroups. As expected, the existence of a natural free product ensures the abundance of new examples and opens up several new paths of research including compactification of a given semihypergroup and amalgamation of different kinds of semihypergroups in pursuit of specific examples and counter-examples among others. The rest of the paper is designed as below.
	

	\quad In the next, \textit{i.e.}, second section of this article, we recall some preliminary definitions and notations given by Jewett in \cite{JE}, and introduce some new definitions required for further work. We conclude the section with listing some important useful examples of semihypergroups and hypergroups. In the third section, we initiate the study of a free structure on semihypergroups. We introduce a free product structure and construct a specific topology and convolution for a family of semihypergroups such that the resulting semihypergroup abides by an universal property equivalent to the universal property for free products in the classical theory of topological groups, thus providing us with a whole new class of examples for semihypergroups.

	
	\quad Finally, this text is the second installment of a systematic study of this unifying category of objects with an intriguing scope for applications in different fields of research discussed above. Hence we conclude the text with some related potential problems and areas which we intend to work on and explore further in near future. Furthermore, introducing and investigating similar structures as in  \cite{MI,DA,FO2,FO3,GH,GH1,LA1,LDS,MIT,NA,PI,RO} for the case of semihypergroups and thereby realising further where and why this theory deviates from the classical theory of topological semigroups and groups, potentially provides us with deeper insight into different questions in the study of abstract harmonic analysis.
	

	
	\section{Preliminary}
	\label{Preliminary}
	
	\noi Here we first list some of the preliminary set of basic notations that we will use throughout the text, followed by a brief introduction to the tools and concepts needed for the formal definition of a semihypergroup (referred to as `semiconvo' in \cite{JE}). Finally, we list some of the well-known natural examples of a semihypergroup \cite{CB,JE,ZE}. All the topologies throughout this text are assumed to be Hausdorff, unless otherwise specified.
	
	\vspace{0.03in}
	
	\noi For any locally compact Hausdorff topological space $X$, we denote by $M(X)$ the space of all regular complex Borel measures on $X$, where $M^+(X)$ and $P_c(X)$ denote the subsets of $M(X)$ consisting of all finite non-negative regular Borel measures on $X$ and all probability measures with compact support on $X$ respectively. Moreover, $B(X), C(X), C_c(X)$ and $C_0(X)$ denote the function spaces of all bounded functions, bounded continuous functions, compactly supported continuous functions and continuous functions vanishing at infinity on $X$ respectively.

	\vspace{0.03in}
	
	\noi Next, we introduce two very important topologies on the positive measure space and the space of compact subsets for any locally compact topological space $X$. Unless mentioned otherwise, we will always assume these two topologies on the respective spaces.
	
	\vspace{0.03in}
	
	\noi The \textbf{cone topology} on $M^+(X)$ is defined as the weakest topology on $M^+(X)$ for which the maps $ \mu \mapsto \int_X f \ d\mu$ is continuous for any $f \in C_c^+(X)\cup \{1_X\}$ where $1_X$ denotes the characteristic function of $X$. Note that if $X$ is compact then it follows immediately from the Riesz representation theorem that the cone topology coincides with the weak*-topology on $M^+(X)$ in this case.

	\vspace{0.03in}
	
	\noi We denote by $\mathfrak{C}(X)$ the set of all compact subsets of $X$. The \textbf{Michael topology} on $\mathfrak{C}(X)$ is defined to be the topology generated by the sub-basis $\{\mathcal{C}_U(V) : U,V \textrm{ are open sets in } X\}$, where
	$$\mathcal{C}_U(V) = \{ C\in \mathfrak{C}(X) : C\cap U \neq \O, C\subset V\}.$$
	
	\noi Note that $\mathfrak{C}(X)$ actually becomes a locally compact Hausdorff space with respect to this natural topology. Moreover if $X$ is compact then $\mathfrak{C}(X)$ is also compact. 
	
	\vspace{0.03in}
	
	\noi For any locally compact Hausdorff space $X$ and any element $x\in X$, we denote by $p_x$ the point-mass measure or the Dirac measure at the point $\{x\}$.

	\vspace{0.03in}
	
	\noi For any three locally compact Hausdorff spaces $X, Y, Z$, a bilinear map $\Psi : M(X) \times M(Y) \rightarrow M(Z)$ is called \textbf{positive continuous} if the following holds :
	\begin{enumerate}
		\item $\Psi(\mu, \nu) \in M^+(Z)$ whenever $\mu \in M^+(X), \nu \in M^+(Y)$.
		\item The map $\Psi |_{M^+(X) \times M^+(Y)}$ is continuous.
	\end{enumerate}

	
	\noi Now we are ready to state the formal definition for a semihypergroup. Note that we follow Jewett's notion \cite{JE} in terms of the definitions and notations, in most cases.
	

	\begin{definition}\label{shyper}\textbf{(Semihypergroup)} A pair $(K,*)$ is called a (topological) semihypergroup if they satisfy the following properties:
		
		
		\begin{description} 
			\item[(A1)] $K$ is a locally compact Hausdorff space and $*$ defines a binary operation on $M(K)$ such that $(M(K), *)$ becomes an associative algebra.
			
			\item[(A2)] The bilinear mapping $* : M(K) \times M(K) \rightarrow M(K)$ is positive continuous.
			
			\item[(A3)] For any $x, y \in K$ the measure $p_x * p_y$ is a probability measure with compact support.
			
			\item[(A4)] The map $(x, y) \mapsto \mbox{ supp}(p_x * p_y)$ from $K\times K$ into $\mathfrak{C}(K)$ is continuous.
		\end{description}
	\end{definition}
	
	
	\noi Note that for any $A,B \subset K$ the convolution of subsets is defined as the following:
	$$A*B := \cup_{x\in A, y\in B} \ supp(p_x*p_y)  .$$

	\noi  We define the concept of left (or right) topological and semitopological semihypergroups, similar to the concepts of left (or right) topological and semitopological semigroups.
	
	
	\begin{definition}
		A pair $(K, *)$ is called a left topological semihypergroup if it satisfies all the conditions of Definition \ref{shyper} with property \textbf{(A$2$)} replaced by the following:
		\begin{description}
			\item[(A2$'$)] The map $(\mu, \nu) \mapsto \mu*\nu$ is positive and for each $\omega \in M^+(K)$ the map $L_\omega:M^+(K) \rightarrow M^+(K)$ given by $L_\omega(\mu)= \omega*\mu$ is continuous.
		\end{description}
	\end{definition}
	
	
	\begin{definition}
		A pair $(K, *)$ is called a right topological semihypergroup if it satisfies all the conditions of Definition \ref{shyper} with property \textbf{(A$2$)} replaced by the following:
		\begin{description}
			\item[(A2$''$)] The map $(\mu, \nu) \mapsto \mu*\nu$ is positive and for each $\omega \in M^+(K)$ the map $R_\omega:M^+(K) \rightarrow M^+(K)$ given by $R_\omega(\mu)= \mu*\omega$ is continuous.
		\end{description}
	\end{definition}

	
	\begin{definition}
		A pair $(K, *)$ is called a \textbf{semitopological semihypergroup} if it is both left and right topological semihypergroup.
	\end{definition}

	\noi Now we list some well known examples \cite{JE,ZE} of semihypergroups and hypergroups. Please see \cite[Section 3]{CB} for details on the constructions as well as why most of the structures discussed there, although attain semihypergroup structures, fail to be hypergroups.  
	
	
	\begin{example} \label{extr}

		If $(S, \cdot)$ is a locally compact topological semigroup, then $(S, *)$ is a semihypergroup where $p_x*p_y = p_{x.y}$ for any $x, y \in S$. Similarly, if $(G, \cdot)$ is a locally compact topological group, then $(G, *)$ is a hypergroup with the same bilinear operation $*$, identity element $e$ where $e$ is the identity of $G$ and the involution on $G$ defined as $x \mapsto x^{-1}$.

	\end{example}

	\begin{example} \label{ex2}
		Take $T = \{e, a, b\}$ and equip it with the discrete topology. Define
		\begin{eqnarray*}
			p_e*p_a &=& p_a*p_e \ = \ p_a\\
			p_e*p_b &=& p_b*p_e \ = \ p_b\\
			p_a*p_b &=& p_b*p_a \ = \ z_1p_a + z_2p_b\\
			p_a*p_a &=& x_1p_e + x_2p_a + x_3p_b \\
			p_b*p_b &=& y_1p_e + y_2p_a + y_3p_b
		\end{eqnarray*}
		
		\noi where $x_i, y_i, z_i \in \mathbb{R}$ such that $x_1+x_2+x_3 = y_1+y_2+y_3 = z_1+z_2 = 1$ and $y_1x_3 = z_1x_1$. Then $(T, *)$ is a commutative hypergroup with identity $e$ and the identity function on $T$ taken as involution.
	\end{example}

	\begin{example} \label{ex3}
		Let $G$ be a locally compact topological group and $H$ be a compact subgroup of $G$. Also, $\mu$ be the normalized Haar measure of $H$. Consider the left coset space $S := G/H = \{xH : x \in G\}$ and the double coset space $K := G/ /H = \{HxH : x \in G\}$ and equip them with the respective quotient topologies. Then $(S, *)$ is a semihypergroup and $(K, *)$ is a hypergroup where the convolutions are given as following for any $x, y \in G$ :
		$$p_{xH} * p_{yH} = \int_H p_{(xty)H} \ d\mu(t), \ \ \ \ p_{HxH} * p_{HyH} = \int_H p_{H(xty)H} \ d\mu(t)  .$$
		
	\end{example}

	\begin{example} \label{ex4}
		Let $G$ be a locally compact topological group and $H$ be any compact group. 
		For any continuous affine action \cite{JE} $\pi$ of $H$ on $G$, consider the orbit space $\mathcal{O} := \{x^H : x \in G\}$, where $x^H = \mathcal{O}(x) = \{\pi(h, x): h\in H\}$. 
		
		
		\noi Let $\sigma$ be the normalized Haar measure of $H$. Consider $\mathcal{O}$ with the quotient topology and the following convolution
		$$ p_{x^H} * p_{y^H} := \int_H \int_H p_{(\pi(s, x)\pi(t, y))^H} \ d\sigma(s) d\sigma(t) .$$
		\noi Then $(\mathcal{O}, *)$ becomes a semihypergroup.
	\end{example}
	
	
	\section{Free Structure on Semihypergroups}
	\label{Free products}
	
	\noi Free groups and free products of topological groups and semigroups have been an useful tool for a number of reasons. Providing specific examples or counter-examples, studying subgroups and quotients of a topological group, constructing an unified group with specific properties of a number of groups and studying the problem of (topologically) embedding a topological semigroup into a topological group are a few of the areas where the use of a free structure proves to be very helpful.


	\noi In this section, we initiate the study of a free structure on semihypergroups. We introduce a free product structure and a specific topology and convolution to a family of semihypergroups such that the resulting semihypergroup abides by an universal property equivalent to the universal property for the free product of a family of topological groups \cite{MO}.  But before we proceed to the definition and construction of a free product, let us first briefly recall some basic formal structures on an arbitrary set and examine how a homomorphism between two semihypergroups translates to a specific homomorphism between their measure algebras, in the classical sense.
	
	\vspace{0.03in}
	
	\noi Recall \cite{CB} that given two (semitopological) semihypergroups $K$ and $H$, a continuous map $\phi:K\ra H$ is called a homomorphism if for any Borel measurable function $f$ on $H$ and for any $x, y \in K$ we have that
	$$ f\circ \phi(x*y) = f(\phi(x)*\phi(y))  .$$

	\vspace{0.03in}
	
	\noi Let $A$ be any nonempty set. Then a \textbf{word} or string on $A$ is a finite sequence $a_1 a_2\ldots a_n$ which is either void or $a_i \in A$ for each $1\leq i\leq n$. Similarly, given a family $\{A_\alpha\}$ of nonempty sets, a word or string on $\{A_\alpha\}$ is a finite sequence $a_1 a_2\ldots a_n$ which is either void or $a_i \in A_\alpha$ for some $\alpha$, for each $1\leq i\leq n$. For convenience of notation, here we sometimes write the word $a_1 a_2\ldots a_n$ also as $(a_1a_2\ldots a_n)$.
	
	
	\begin{definition}
		Let $\{K_\alpha\}$ be a family of semihypergroups and $w= a_1a_2\ldots a_n$ is a word on $\{K_\alpha\}$ where $a_i \in K_{\alpha_i}$ for $1\leq i \leq n$. Then $w$ is called a \textbf{reduced word} if either $w$ is empty or if $w$ satisfies the following two conditions:
		\begin{enumerate}
			\item $a_i \neq 1_{K_\alpha}$ for $1\leq i \leq n$ where $1_{K_\alpha}$ is the identity element of $K_\alpha$ (if identity exists in $K_\alpha$).
			\item ${\alpha_i} \neq {\alpha_{i+1}}$ for $1\leq i \leq n-1$.
		\end{enumerate}
	\end{definition}
	
	
	\noi Given a reduced word $a_1 a_2 \ldots a_n$ on a family of semihypergroups $\{K_\alpha\}$, the \textbf{length} of the word is defined to be the number of elements in the word, and is written as $\mathit{l}(a_1 a_2 \ldots a_n) = n$.
	
	
	\begin{proposition}
		Given a finite set $A$ of cardinality $n$, we can form the free semihypergroup $(F_A, *)$ generated by $A$, where $F_A$ is the set of all words (finite sequences) consisting of $0$ or more elements from $A$.
	\end{proposition}
	\begin{proof}
		Given any finite set of cardinality $n$, we can easily form a free semigroup $(F, \cdot)$ generated by $A$ where  where $F$ is the set of all words (finite sequences) consisting of $0$ or more elements from $A$ and the multiplication of two words $a_1a_2\ldots a_n$, $b_1b_2\ldots b_m$ are simple concatenation given by
		$$ (a_1a_2\ldots a_n). (b_1b_2\ldots b_m) = a_1a_2\ldots a_nb_1b_2\ldots b_m   .$$
		\noi Note that equipping $F$ with the discrete topology makes it into a topological semigroup. Then as outlined in Example \ref{extr} we can consider $(F, \cdot)$ as a semihypergroup $(F_A, *)$ where as sets, we have $F=F_A$.
		\qed
	\end{proof}
	
	
	\begin{theorem} \label{homtopc}
		Let $K$ and $H$ be two semihypergroups. Then a homomorphism $\phi:K \ra H$ induces a positive continuous homomorphism $\Gamma_\phi: M(K) \ra M(H)$.
	\end{theorem}
	\begin{proof}
		\noi Define $T_\phi: C_0(H) \ra C_0(K)$ by
		$$T_\phi f(x) := f(\phi(x))$$
		\noi for each $f\in C_0(H), x\in K$.
		
		\vspace{0.03in}
		
		\noi Then we know that the adjoint operator $T_\phi^*: C_0(K)^* \ra C_0(H)^*$ is given by
		$$T_\phi^* (m)(f) := m(T_\phi f)$$
		\noi for each $m\in C_0(K)^*, f \in C_0(H)$.
		
		\vspace{0.03in}

		\noi We define $\Gamma_\phi:= T_\phi^*:M(K) \ra M(H)$. Hence on point-mass measures $p_x$ for any $x\in K$, we have
		$$\Gamma_\phi(p_x) (f) = f(\phi(x)) = p_{\phi(x)}(f)$$
		\noi for any $f\in C_0(H)$.

		\vspace{0.03in}
		
		\noi Hence we have $\Gamma_\phi$ in the following way. Define
		$$\Gamma_\phi(p_x):=p_{\phi(x)}$$
		\noi and then extend $\Gamma_\phi$ linearly to $M_F^+(K)$.  Since $M_F^+(K)$ is dense in $M^+(K)$ we can extend $\Gamma_\phi$ to $M^+(K)$ and then finally we extend $\Gamma_\phi$ linearly to $M(K)$.
		
		\vspace{0.03in}
		
		\noi Hence by construction, we get that $\Gamma_\phi$ is a positive continuous linear map on $M(K)$. Now to see if it is a homomorphism, it suffices to show that 
		$$\Gamma_\phi(p_x*p_y) = \Gamma_\phi(p_x)*\Gamma_\phi(p_y)$$
		for any $x, y \in K$. 
		
		\vspace{0.03in}

		\noi Now pick any $x, y \in K$. Then $p_x*p_y\in M^+(K)$. Hence there exists a net $\{m_\alpha\}$ in $M_F^+(K)$ such that 
		$$m_\alpha \ra (p_x*p_y) .$$

		
		\noi For each $\alpha$ let 
		$$m_\alpha = \sum_{i=1}^{k_\alpha}  \ c_i^\alpha p_{x_i^\alpha}$$
		for some $k_\alpha \in \mathbb{N}, c_i^\alpha \in \mathbb{F}, x_i^\alpha \in K$. Thus we have
		\begin{eqnarray*}
			\Gamma_\phi(p_x*p_y) &=& \Gamma_\phi(\lim_\alpha m_\alpha)\\
			&=& \lim_\alpha \Gamma_\phi\big{(}\sum_{i=1}^{k_\alpha}  \ c_i^\alpha p_{x_i^\alpha}\big{)}\\
			&=& \lim_\alpha \sum_{i=1}^{k_\alpha}  \ c_i^\alpha \Gamma_\phi(p_{x_i^\alpha})\\
			&=& \lim_\alpha \sum_{i=1}^{k_\alpha}  \ c_i^\alpha p_{\phi(x_i^\alpha)} = \lim_\alpha \ \tilde{m}_\alpha
		\end{eqnarray*}
		\noi  where for each $\alpha$ we set 
		$$\tilde{m}_\alpha := \sum_{i=1}^{k_\alpha}  \ c_i^\alpha p_{\phi(x_i^\alpha)}  .$$
		\noi Now pick any $f\in C_0(H)$. We have
		\begin{eqnarray*}
			(p_{\phi(x)}*p_{\phi(y)})(f) &=& \int_H f \ d(p_{\phi(x)}*p_{\phi(y)})\\
			&=& f(\phi(x) * \phi(y))\\
			&=& (f \circ \phi)(x*y)\\
			&=& \int_K (f \circ \phi) \ d(p_x*p_y)\\
			&=& \int_K (f \circ \phi) \ d(\lim_\alpha m_\alpha) = \lim_\alpha \int_K (f \circ \phi) \ d(m_\alpha)
		\end{eqnarray*}
		\noi where the third equality holds since $\phi$ is a homomorphism.  Thus we have that
		\begin{eqnarray*}
			(p_{\phi(x)}*p_{\phi(y)})(f) &=& \lim_\alpha \sum_{i=1}^{k_\alpha} c_i^\alpha \int_K (f \circ \phi) \ d(p_{x_i^\alpha})\\
			&=& \lim_\alpha \sum_{i=1}^{k_\alpha} c_i^\alpha f(\phi(x_i^\alpha))\\
			&=& \lim_\alpha \sum_{i=1}^{k_\alpha} c_i^\alpha p_{\phi(x_i^\alpha)} (f) = \lim_\alpha \tilde{m}_\alpha(f)  .
		\end{eqnarray*}
		
		\noi Since $f\in C_0(H)$ was chosen arbitrarily, it follows that 
		$$\lim_\alpha \tilde{m}_\alpha =  p_{\phi(x)}*p_{\phi(y)}  .$$
		Hence for any $x, y \in K$ we have that 
		\begin{eqnarray*}
			\Gamma_\phi(p_x*p_y) &=& p_{\phi(x)}*p_{\phi(y)} \\
			&=& \Gamma_\phi(p_x)*\Gamma_\phi(p_y)  .
		\end{eqnarray*} 
		\qed
	\end{proof}

	
	\noi Now we define the free product of a family of (semitopological) semihypergroups, following the standard norms of defining a free product on a category of objects. Note that the definition follows the definition of free product for topological groups \cite[Definition 2.1]{MO}. The first two structural conditions essentially remain the same as in the case of topological groups, whereas the universal property is naturally defined in terms of measure spaces, generalising the product to semihypergroups.

	
	\begin{definition} \label{deffp}
		Let $\{K_\alpha\}$ be a family of semihypergroups. Then a semihypergroup $F$ is called the free (topological) product of $\{K_\alpha\}$, if the following conditions are satisfied:
		
		\begin{enumerate}
			
			\item For each $\alpha$, $K_\alpha$ is a sub-semihypergroup of $F$.
			
			\item $F$ is algebraically generated by $\cup_\alpha K_\alpha$ (in terms of finite reduced words).
			
			\item Given any semihypergroup $H$, if we have a continuous homomorphism $\phi_\alpha: K_\alpha \ra H$ for each $\alpha$, then there exists a unique positive continuous homomorphism $\Gamma: M(F) \ra M(H)$ such that $\Gamma|_{M(K_\alpha)} = \Gamma_{\phi_\alpha}$ for each $\alpha$.
		\end{enumerate}


		\noi We denote $F$ as $\prod^*_\alpha K_\alpha \ $.

	\end{definition}
	
	
	\noi Now before we prove that such products exists for all the significant examples discussed in the third section, let us first define a specific type of semihypergroups.

	
	\begin{definition}
		Let $(K, *)$ be a (semitopological) semihypergroup. We say that $K$ is a \textbf{pure semihypergroup} if for any $x, y \in K\setminus\{e\}$ we have that $supp(p_x*p_y)$ is not singleton whenever $e \in supp(p_x*p_y)$. Here $e$ is the identity of $K$. 
		
		
		\noi Equivalently, we can say that $K$ is a \textbf{pure semihypergroup} if we have that $p_x * p_y \neq p_e$ for any $x, y \in K \setminus \{e\}$.
	\end{definition}

	
	\noi Note that if a semihypergroup does not have an identity, then the above conditions are vacuously true and so $K$ is a pure semihypergroup. Also, note that all the non-trivial examples discussed in Examples  \ref{ex2}, \ref{ex3} and \ref{ex4} are pure semihypergroups. 
	
	\noi However, locally compact groups clearly fail to be a pure semihypergroup. In fact, a locally compact semigroup $(S, \cdot)$ with identity $e$ can be seen as a pure semihypergroup (Example \ref{extr}) only if $x.y \neq e$ for any $x, y \in S\setminus \{e\}$, \textit{i.e,} if none of the elements in $S$ have a left or right inverse. 
	
	
	\begin{theorem}
		Let $\{K_\alpha\}$ be a family of semihypergroups such that if more than one of the $K_\alpha$'s has an identity element, then all the $K_\alpha$'s are pure semihypergroups. Then there exists a semihypergroup $F$ such that $F= \prod^*_\alpha K_\alpha$.
		
		\vspace{0.03in}
		
		\noindent Moreover, if more than one of the $K_\alpha$'s has an identity element, then $F$ is also a semihypergroup with identity.
	\end{theorem}
	\begin{proof}
		Set $\Lambda := \{\alpha : K_\alpha \mbox{ \ has an identity element}\}$. If $|\Lambda| \leq 1$, set
		$$ K := \{ a_1 a_2 \ldots, a_n : n\in \mathbb{N}, a_1 a_2\ldots a_n \mbox{ \ is a reduced word on } \{K_\alpha\}\} .$$
		Note that if $|\Lambda|=1$, then we treat the identity as any other arbitrary element in that semihypergroup. For convenience of notations, we rename the idenity element as $x_e$ in this case, so that it can be included in finite sequences of reduced words.
		
		\vspace{0.03in}

		\noi If $|\Lambda|>1$, let $e$ denote the empty word on $\{K_\alpha\}$ and set $e= 1_{K_\alpha}$ for each $\alpha\in \Lambda$ where $1_{K_\alpha}$ denotes the identity element of $K_\alpha$. Now set
		$$ K := \{e\} \ \cup \  \{ a_1 a_2\ldots a_n : n\in \mathbb{N}, \ a_1 a_2\ldots a_n \mbox{ is a reduced word on } \{K_\alpha\}\}  .$$
		\noi For any $x_1 x_2\ldots x_n$, $y_1 y_2\ldots y_m \in K$ where $x_i \in K_{\alpha_i}, y_j \in K_{\beta_j}$ for $i=1, 2, \ldots, n$; $j=1, 2, \ldots, m$ and any subset $A \subset K_{\alpha_0}$ where $\alpha_0 \neq \alpha_n , \beta_1$ we define that:
		$$(x_1 \ldots x_n) A (y_1 \ldots y_m) := \begin{cases}
		\{x_1 x_2 \ldots x_n a  y_1 y_2 \ldots y_m : a\in A\}, &  \hspace{-0.12in}  \mbox{ if }  \ e \in \hspace{-0.12in} / \ A  .\\
		& \\
		\{(x_1 x_2 \ldots x_n y_1 y_2 \ldots y_m)\} \cap K  & \\
		\cup \{(x_1 x_2 \ldots x_n a  y_1 y_2 \ldots y_m) : a\in A, a\neq e\}, &    \hspace{-0.12in}  \mbox{  if } \ e \in A  .\\
		\end{cases}$$
		\noi Now for any two sets $A_1\subset K_{\alpha_1}$, $A_2\subset K_{\alpha_2}$ where $\alpha_1 \neq \alpha_2$ define
		\begin{eqnarray*}
			A_1A_2 &:=& \{a_1 a_2: a_1 \in A_1, a_2 \in A_2\}, \mbox{ if } e\in \hspace{-0.11in} / \  A_1, A_2 . \\
			A_1A_2 &:=& \{a_1 a_2: a_1 \in A_1\setminus \{e\}, a_2 \in A_2\} \cup \{a_2 : a_2 \in A_2\}, \mbox{ if } e \in A_1, e \in \hspace{-0.11in} / \ A_2  .\\
			A_1A_2 &:=& \{a_1 a_2: a_1 \in A_1, a_2 \in A_2 \setminus \{e\}\} \cup \{a_1 : a_1 \in A_1\}, \mbox{ if } e \in \hspace{-0.11in} / \ A_1, e \in A_2  .\\
			A_1A_2 &:=& \{a_1 a_2: a_1 \in A_1 \setminus \{e\}, a_2 \in A_2 \setminus \{e\}\} \cup \{a_1 : a_1 \in A_1 \setminus \{e\}\}\\ && \cup \{a_2 : a_2 \in A_2 \setminus \{e\}\} \cup \{e\}, \mbox{ if } e \in  A_1,  A_2  .
		\end{eqnarray*}
		\noi Similarly for $A_1\subset K_{\alpha_1}$, $A_2\subset K_{\alpha_2}$, $A_3\subset K_{\alpha_3}$ where $\alpha_1 \neq \alpha_2 \neq \alpha_3$, we define
		\begin{eqnarray*}
			A_1A_2A_3 &:=& \{a_1 a_2 a_3: a_1 \in A_1, a_2 \in A_2, a_3 \in A_3\}, \mbox{ if } e\in \hspace{-0.11in} / \  A_1, A_2, A_3  .\\
			A_1A_2A_3 &:= & \{a_1 a_2 a_3: a_1 \in A_1\setminus \{e\}, a_2 \in A_2, a_3 \in A_3 \}\\
			&& \cup \{a_2 a_3 : a_2 \in A_2, a_3\in A_3\}, \mbox{ if } e \in A_1, \ e \in \hspace{-0.11in} / \ A_2, A_3  .\\
			A_1A_2A_3 &:=& \{a_1 a_2 a_3: a_1 \in A_1\setminus \{e\}, a_2 \in A_2 \setminus \{e\}, a_3 \in A_3\} \\
			&& \cup \{a_2 a_3 : a_2 \in A_2\setminus \{e\}, a_3\in A_3\} \\
			&& \cup \{a_1 a_3: a_1 \in A_1\setminus \{e\}, a_3\in A_3\} \cup \{a_3: a_3\in A_3\}, \\
			&& \mbox{ if } e \in A_1, A_2, \ e \in \hspace{-0.11in} / \  A_3  .\\
			A_1A_2A_3 &:=& \{a_1 a_2 a_3: a_1 \in A_1 \setminus \{e\}, a_2 \in A_2 \setminus \{e\}, a_3\in A_3\setminus \{e\}\} \\
			&& \cup \{a_1 a_2 : a_1 \in A_1 \setminus \{e\}, a_2\in A_2\setminus \{e\}\} \cup \{a_2, a_3 : a_2 \in A_2 \setminus \{e\},\\
			&& a_3\in A_3\setminus \{e\}\}  \cup \{a_1, a_3 : a_1 \in A_1 \setminus \{e\}, a_3\in A_3\setminus \{e\}\} \\
			&& \cup \{a_1: a_1\in A_1\setminus \{e\}\} \cup \{a_2: a_2\in A_2\setminus \{e\}\} \\
			&& \cup \{a_3: a_3\in A_3\setminus \{e\}\} \cup \{e\}, \mbox{ if } e \in  A_1,  A_2, A_3  .
		\end{eqnarray*}
		
		
		\noi We can get the other cases following the same pattern. 
		
		\noi In general for any $n \in \mathbb{N}$, consider subsets $A_i \subset K_{\alpha_i}$ for $i=1, 2, \ldots, n$, where $\alpha_i \neq \alpha_{i+1}$ for $i= 1, 2, \ldots, (n-1)$. Let $e \in A_{i_k}$ for $k= 1, 2, \ldots, m$  where $i_{k_1}<i_{k_2}$ for $k_1<k_2$ and $e \in \hspace{-0.11in} / \ A_i$ whenever $i \neq i_k$ for any $k= 1, 2, \ldots, m$. 
		
		\vspace{0.03in}
		
		\noi Then continuing in the same manner as above, we get that
		$$ A_1A_2\ldots A_n := \begin{cases}
		\{ a_1 a_2\ldots a_n \in K : a_i \in A_i, i = 1, 2, \ldots, n\} & \\
		\cup \ ( F_n \cap K ) \cup \{e\} \ , & \mbox{ if } m=n \\
		&\\
		\{ a_1 a_2\ldots a_n \in K : a_i \in A_i, i = 1, 2, \ldots, n\} &\\
		\cup \ ( F_n \cap K ) \ , & \mbox{ if } m\neq n\\
		
	\end{cases}$$
	\noi where $F_n$ is defined as the following:
	\begin{equation*}
	\begin{split}
	F_n := & \bigcup_{l=1}^m \bigcup_{\small{\substack{\hspace{0.05in}1\leq k_1<k_2<\ldots\\ \ldots<k_l\leq m}}} \hspace{-0.05in} \Big{\{ }  (a_1 a_2 \ldots a_{{i_{k_1}-1}} a_{{i_{k_1}}+1} \ldots a_{i_{k_2}-1} a_{{i_{k_2}}+1} \ldots a_{i_{k_l}-1} a_{i_{k_l}+1} \ldots \\[-3.5ex]
	&  \hspace{2.1in} \ldots a_{n-1}a_n) : a_j \in A_j, j= 1, 2, \ldots, n \Big{\}}  .
	\end{split}
	\end{equation*}
	\noi Note that from now on unless otherwise mentioned, whenever we write a product $A_1A_2\ldots A_n$ of subsets $A_i \subset K_{\alpha_i}$, we assume that $\alpha_i \neq \alpha_{i+1}$ for $i=1, 2, \ldots, (n-1)$. Also it follows immediately from the above construction that for any such sets $A_1, A_2, \ldots, A_n$ we have
	$$ A_1A_2\ldots A_n = A_1(A_2A_3\ldots A_n) = A_1(A_2A_3\ldots A_{n-1})A_n = (A_1A_2\ldots A_{n-1})A_n  .$$

	
	\paragraph{\textbf{Natural Topology on $K$:}}
	
	\hspace{0.03in}
	
	\noi Now we first introduce a certain topology $\tau$ on $K$. Afterwards we will define a binary operation `$*$' on the measure space $M(K)$  and then investigate how the operation gives rise to an associative algebra on $M(K)$ and how the Borel sets act under the said operation and the cone topology on $M^+(K)$ induced by $\tau$. First, we set 
	$$\mathcal{B}:= \{U_1U_2, \ldots, U_n : U_i \mbox{ is an open subset of } K_{\alpha_i}, i=1,2, \ldots n;  n\in \mathbb{N}\}  .$$
	\noi We need to show that $\mathcal{B}$ serves as a base for a topology on $K$. Pick any element $x_1 x_2\ldots x_n \in K\setminus \{e\}$ where $x_i \in K_{\alpha_i}$, $i=1,2, \ldots, n$. Since $x_1 x_2\ldots x_n$ is a reduced word, $x_i \neq 1_{K_{\alpha_i}}$, even if $\alpha_i \in \Lambda$. Thus we can find an open neighbourhood $V_{x_i}$ of $x_i$ in $K_{\alpha_i}$ such that $e\in \hspace{-0.11in} / \ V_{x_i}$. Then $x_1 x_2\ldots x_n \in V_{x_1}V_{x_2}\ldots V_{x_n} \in \mathcal{B} \ $. 
	
	\vspace{0.05in}
	
	\noi Also note that if $e\in K$, then for any open neighborhood $U$ of $e=1_{K_\alpha}$ in $K_\alpha$ where $\alpha \in \Lambda$, we have that $e\in U \in \mathcal{B}$. Now let $U_1U_2\ldots U_n, V_1V_2\ldots V_m \in \mathcal{B}$ where $U_i\subset K_{\alpha_i}$ and $V_j \subset K_{\beta_j}$ for $i=1, 2, \ldots n$; $j=1, 2, \ldots, m$. Pick any element $x_1 x_2 \ldots x_l \in U_1U_2\ldots U_n \cap V_1V_2\ldots V_m$. Note that $l \leq \min({n, m})$ and hence for each $k=1, 2, \ldots l$, there exists some $i_k \in \{1, 2, \ldots, n\}$ and $j_k \in \{1, 2, \ldots, m\}$ such that $i_1 < i_2 < \ldots < i_l$, \ $j_1 < j_2 < \ldots < j_l$ and $x_k \in U_{i_k} \cap V_{j_k}$ for each $k=1, 2, \ldots, l \ $. 
	
	\vspace{0.05in}
	
	\noi Set $W_k:=U_{i_k} \cap V_{j_k}$ for each $k=1, 2, \ldots, l$. Since each $W_k$ is open in $K_{\alpha_{i_k}}=K_{\beta_{j_k}}$ and $x_k \neq e$ for each $k$, we have that $x_1 x_2 \ldots x_l \in W_1W_2\ldots W_l \in \mathcal{B}$. Now to show that $W_1W_2\ldots W_l \subset U_1U_2\ldots U_n \cap V_1V_2\ldots V_m$, pick any $y_1 y_2 \ldots y_t \in W_1W_2\ldots W_l$. Then for $k = 1, 2, \ldots t$, we have $y_k \in W_{s_k} = U_{i_{s_k}}\cap V_{j_{s_k}}$, for some $s_k\in \{1, 2, \ldots, l\} \ $.
	
	\vspace{0.05in}
	
	\noi Since $x_1 x_2 \ldots x_l \in U_1U_2\ldots U_n$ we must have that $e\in U_i$ whenever $i\in \{1,2, \ldots n\} \setminus \{i_k : k= 1, 2, \ldots, l\}$.  Again since  $y_1 y_2\ldots y_t \in W_1W_2\ldots W_l$, we must have that $e\in W_{s} = U_{i_{s}}\cap V_{j_{s}}\subset U_{i_s}$ whenever $s\in \{1, 2, \ldots, l\}\setminus \{s_k: k=1, 2, \ldots, t\}$. Combining these two statements gives us that $e\in U_i$ whenever $i\in \{1, 2, \ldots, n\} \setminus \{i_{s_k}: k=1, 2, \ldots, t\}$. Also, we know that $y_k \in U_{i_{s_k}}$ for each $k= 1, 2, \ldots, t$. Hence it follows that $y_1 y_2 \ldots y_t \in U_1U_2\ldots U_n \ $.
	
	\vspace{0.05in}
	
	\noi Similarly since $x_1 x_2 \ldots x_l \in V_1V_2\ldots V_m$, we must have that $e\in V_j$ whenever $j\in \{1,2, \ldots m\} \setminus \{j_k : k= 1, 2, \ldots, l\}$ and since $y_1 y_2 \ldots y_t \in W_1W_2\ldots W_l$, we must have that $e\in W_{s} = U_{i_{s}}\cap V_{j_{s}}\subset V_{j_s}$ whenever $s\in \{1, 2, \ldots, l\}\setminus \{s_k: k=1, 2, \ldots, t\}$, we can proceed in the same manner to get that $y_1 y_2 \ldots y_t \in V_1V_2\ldots V_m \ $.
	
	\vspace{0.05in}
	
	\noi Thus finally we see that $x_1 x_2 \ldots x_l \in W_1W_2\ldots W_l \subset \big{(}U_1U_2\ldots U_n \cap V_1V_2\ldots V_m\big{)}$ where $W_1W_2\ldots W_l \in \mathcal{B}$, implying that $\mathcal{B}$ serves as a base for a topology on $K$. We denote the topology on $K$ generated by $\mathcal{B}$ as $\tau$.
	
	\vspace{0.05in}
	
	\paragraph{\textbf{Natural Binary Operation $*$ on $M(K)$:}}
	
	\hspace{0.03in}

	\noi As pointed out in the first section of this text, in order to define a binary operation on $M(K)$ it is sufficient to define a binary operation on the point-mass measures on $M(K)$. First of all, if $\Lambda\neq \varnothing$ then for any $x_1 x_2\ldots x_n\in K$ we define
	$$p_e*p_e=p_e \mbox{ \ \ ; \ \ } \ p_e*p_{x_1 x_2\ldots x_n} = p_{x_1 x_2\ldots x_n} = p_{x_1 x_2\ldots x_n}*p_e  .$$
	
	\noi Now for any two elements $x=x_1 x_2\ldots x_n, y=y_1 y_2\ldots y_m \in K\setminus \{e\}$ where 
	$x_i \in K_{\alpha_i}, y_j \in K_{\beta_j}$ for $i=1, 2, \ldots, n$; \ $j=1, 2, \ldots, m$ and for any Borel set $E' \subseteq K$ we define:
	
	$$(p_x*p_y) (E') = \begin{cases}
	p_{(x_1 \ldots x_n y_1 \ldots y_m)}(E')  , &  \hspace{-0.05in} \mbox{ if } \alpha_n \neq \beta_1.\\
	& \hspace{-0.05in} \\
	(p_{x_n}*p_{y_1})(E)   ,& \hspace{-0.05in} \mbox{  if } \alpha_n = \beta_1, \alpha_{n-1}\neq\beta_2\\
	& \hspace{-0.05in} \mbox{ and } E' \cap \big{(}(x_1 \ldots x_{n-1})K_{\alpha_n}(y_{2} \ldots y_m)\big{)}\\
	& \hspace{-0.05in} \ = (x_1 \ldots x_{n-1})E(y_{2} \ldots y_m).\\
	& \hspace{-0.05in} \\
	(p_{x_n}*p_{y_1})(E\cup\{e\})  , &  \hspace{-0.05in} \mbox{ if } \alpha_n = \beta_1, \alpha_{n-1}=\beta_2\\
	& \hspace{-0.05in} \mbox{ and } E' \cap \big{(}(x_1 \ldots x_{n-1})K_{\alpha_n}(y_{2} \ldots y_m)\big{)}\\
	& \hspace{-0.05in}  \ = (x_1 \ldots x_{n-1})E(y_{2} \ldots y_m) \neq \varnothing.\\
	& \hspace{-0.05in} \\
	0  , & \hspace{-0.05in}  \mbox{  if } \alpha_n = \beta_1, \alpha_{n-1}=\beta_2\\
	& \hspace{-0.05in} \mbox{ and } E' \cap \big{(}(x_1 \ldots x_{n-1})K_{\alpha_n}(y_{2} \ldots y_m)\big{)}\\
	& \hspace{-0.05in} \ =\varnothing.\\
	\end{cases}$$

	
	\noindent Note that when $\alpha_n = \beta_1$, then $(p_x*p_y)$ mimics the behaviour of the probability measure $(p_{x_{n}}*p_{y_{1}})$ and hence is a probability measure with support $(x_1 \ldots x_{n-1})\mbox{ supp}(p_{x_{n}}*p_{y_{1}})(y_{2} \ldots y_m)$. 
	
	
	\paragraph{\textbf{$(p_x*p_y)$ has compact support:}}
	
	\hspace{0.03in}
	
	\noi Due to the construction of $\tau$, for any two closed subsets $C_1 \subset K_{\alpha_1}$ and $C_2 \subset K_{\alpha_2}$ where $\alpha_1\neq \alpha_2$, we have that $C_1C_2$ is a closed subset of $K$. This is true since $K_{\alpha_1} K_{\alpha_2}$ is open in $K$ and we have that
	$$K_{\alpha_1}K_{\alpha_2} \setminus C_1C_2 = (K_{\alpha_1}\setminus C_1).K_{\alpha_2} \cup K_{\alpha_1}.(K_{\alpha_2}\setminus C_2),$$
	\noindent which is open in $K_{\alpha_1} K_{\alpha_2}$. Similarly, we can show that $C_1C_2\ldots C_n$ is a closed set in $K$ for any $n\in \mathbb{N}$ where $C_i \in K_{\alpha_i}$ for $1\leq i \leq n$ and $\alpha_i \neq \alpha_{i+1}$ for $1\leq i \leq n-1$. In fact, it is now easy to see that the set $\{C_1C_2\ldots C_n : n\in \mathbb{N}, C_i \mbox{ \ is a closed set in some \ } K_{\alpha_i}\}$ is the family of basic closed sets in $K$.

	\vspace{0.05in}
	
	\noi Now let both $C_1$ and $C_2$ are compact subsets of $K_{\alpha_1}$ and $K_{\alpha_2}$ respectively and let $\{A_\gamma\}_{\gamma\in I}$ be a family of closed subsets of $C_1C_2$ with finite intersection property. Then we must have that $A_\gamma = B_1^\gamma B_2^\gamma$ for some closed subsets $B_i^\gamma \subset C_i \subset K_{\alpha_i}$ for each $\gamma$, $i=1, 2$. Now for any $\{\gamma_1, \gamma_2, \ldots, \gamma_n\} \subset I$ we have that $\cap_1^n B_1^{\gamma_i} B_2^{\gamma_i} = \cap_1^n A_{\gamma_i} \neq \varnothing$.

	\vspace{0.05in}
	
	\noi Hence in particular, $\cap_1^n B_1^{\gamma_i} \neq \varnothing$ for any such finite subset $\{\gamma_1, \gamma_2, \ldots, \gamma_n\} \subset I$. So $\{B_1^\gamma\}_{\gamma\in I}$ is a family of subsets in $C_1$ with the finite intersection property. But $C_1$ is compact in $K_{\alpha_1}$ and hence there exists some $x_1 \in C_1$ such that $x_1 \in \cap_{\gamma\in I} B_1^\gamma$. Similarly, since $C_2$ is also compact in $K_{\alpha_2}$, there exists some $x_2 \in C_2$ such that $x_2 \in \cap_{\gamma\in I} B_2^\gamma$. Thus we see that there exists an element $x \in C_1C_2$ such that $x\in \cap_{\gamma \in I} A_\gamma$ where $x$ is given as the following:
	$$x = \begin{cases}
	(x_1x_2)  , & \mbox{  if } x_1, x_2 \neq e .\\
	(x_1)  , & \mbox{  if } x_1\neq e, x_2 = e  .\\
	(x_2)  , & \mbox{  if } x_1 =e, x_2 \neq e  .\\
	e  , & \mbox{  if } x_1= x_2 = e  .
	\end{cases}$$
	\noi Thus we see that $C_1C_2$ has to be compact too. Now let $C_1,C_2, C_3$ be compact subsets of $K_{\alpha_1},K_{\alpha_2}, K_{\alpha_3}$ respectively. Note that if $\alpha_i \neq \alpha_j$ for any $i, j \in \{1, 2, 3\}$ or if $e\not\in  \ \cap_{i=1}^3 C_i$ then we can proceed similarly as above to show that $C_1C_2C_3$ gives us a compact set in $K$. Otherwise, first consider the following case.
	
	\vspace{0.05in}

	\noi Let $C_0:=(a_1\ldots a_s)C(b_1 \ldots b_t)$, where $C$ is a compact set in $K_{\alpha_0}$ and $ (a_1a_2\ldots a_s),\\ (b_1 b_2 \ldots b_t) \in K \setminus \{e\}$ where $a_i\in K_{\gamma_i}$ and $b_j\in K_{\delta_j}$ for $1\leq i \leq s, 1\leq j \leq t$ such that $\alpha_0 \neq \gamma_s, \delta_1$. We claim that any such $C_0$ is compact in $(K, \tau)$. 
	
	\vspace{0.05in}
	
	\noi Let $\{U_\alpha\}$ be any basic open cover of $C_0$. Hence $U_\alpha = U_1^\alpha U_2^\alpha\ldots U_{n_\alpha}^\alpha$ for some $n_\alpha\in \mathbb{N}$ where each $U_k^\alpha$ is an open set in some $K_\alpha$ such that the product exists. Now for each $x \in C\setminus \{e\}$ there exists $\alpha_x$ such that $(a_1\ldots a_s x b_1 \ldots b_t) \in U_{\alpha_x}$. Therefore there exists $i_x\in \{1, 2, \ldots, n_{\alpha_x}\}$ such that $x\in U_{i_x}^{\alpha_x}$ where $U_{i_x}^{\alpha_x}$ is an open set in $K_{\alpha_0}$. Thus $\mathcal{U}:= \{U_{i_x}^{\alpha_x}\}_{x\in C}$ serves an open cover of the set $C\setminus \{e\}$. 
	
	\vspace{0.05in}
	
	\noi If $e \not\in C$, then $\mathcal{U}$ serves as an open cover of $C$ and hence has a finite subcover $\{U_1, U_2, \ldots, U_p\}$ where $U_k = U_{i_{x_k}}^{\alpha_{x_k}}$ for some $x_k \in C$ for $1\leq k \leq p$. Thus $\{U_{\alpha_{x_1}}, U_{\alpha_{x_2}}, \ldots, U_{\alpha_{x_p}}\}$ serves as a finite subcover of $\{U_\alpha\}$.
	
	\vspace{0.05in}
	
	\noi Now let $e\in C$. Pick $\alpha' \in \Lambda$ such that $\alpha' \neq \alpha_0$. Since $K_{\alpha'}$ is locally compact, we can find a compact neighborhood $F$ of $e$ in $K_{\alpha'}$. Let $\mathcal{V}:=\{V_\beta\}$ be an open cover of $F$. Then $C\cup F$ is a compact set in $K_{\alpha_0}\cup K_{\alpha'}$ with open cover $\mathcal{U} \cup \mathcal{V}$. Thus we get a finite subcover of the form $\{U_1, U_2, \ldots, U_p, V_{\beta_1}, V_{\beta_2}, \ldots V_{\beta_q}\}$ where $U_k = U_{i_{x_k}}^{\alpha_{x_k}}$ for some $x_k \in C\setminus \{e\}$ for $1\leq k \leq p$. Note that $e\in V_{\beta_l}$ for some $l\in \{1, 2, \ldots, q\}$ here and hence $\{U_{i_{x_k}}\}_{k=1}^p$ covers $C\setminus \{e\}$. 
	
	\vspace{0.05in}
	
	\noi Thus if $\gamma_s=\delta_1$, then $\{U_{\alpha_{x_1}}, U_{\alpha_{x_2}}, \ldots, U_{\alpha_{x_p}}\}$ covers $C_0$, as desired. Otherwise if $\gamma_s\neq \delta_1$, then there exists some $\alpha'$ such that $(a_1\ldots a_sb_1\ldots b_t) \in U_{\alpha'}$. Hence in this case, $\{U_{\alpha_{x_1}}, U_{\alpha_{x_2}}, \ldots, U_{\alpha_{x_p}}, U_{\alpha'}\}$ serves as a finite subcover of $C_0$. Hence we see that any set of the form $C_0$ is compact in $(K, \tau)$. Using similar argument we can immediately see that sets of the form $C_1C_2C_3$ are also compact where $C_i \in K_{\alpha_i}$ is compact.

	\vspace{0.05in}
	
	\noi In particular, for any two words $x = (x_1x_2\ldots x_n), y = (y_1y_2\ldots y_m)\in K$ where $x_i \in K_{\alpha_i}$, $y_j \in K_{\beta_j}$ and $\alpha_n = \beta_1$, we have that $$\mbox{ supp}(p_x*p_y) = (x_1 \ldots x_{n-1})\mbox{ supp}(p_{x_{n}}*p_{y_{1}})(y_{2} \ldots y_m)$$ is compact since $\mbox{ supp}(p_{x_{n}}*p_{y_{1}})$ is compact in $K_{\alpha_{n}} = K_{\beta_{1}}$. Thus it follows from the construction of convolution products that for any two words $x, y \in K$ we have that $(p_x*p_y)$ is a probability measure with compact support. Also since convolution is associative in each $M(K_\alpha)$ and concatenation of finite words is associative, it follows from the construction that convolution products on $M(K)$ are associative.

	
	\paragraph{\textbf{$\tau$ is Hausdorff:}}
	
	\hspace{0.03in}

	\noi Now we proceed further to show that $(K, *)$ indeed gives us a semihypergroup. First note that since each $K_\alpha$ is Hausdorff, for any two elements $x_1 x_2\ldots x_n, y_1 y_2\ldots y_m \in K$ where $x_i \in K_{\alpha_i}$, $y_j \in K_{\beta_j}$, we can find sets $A_i\subset K_{\alpha_i}$, $B_j \subset K_{\beta_j}$ for $1\leq i\leq n$, $1\leq j \leq m$ such that $x_i \in A_i$, $y_j \in B_j$ and $A_i \cap B_j = \varnothing$ for all $i, j$. Hence we have that $x_1 x_2\ldots x_n\in A_1A_2\ldots A_n, \ y_1 y_2\ldots y_m\in B_1B_2\ldots B_m$ where $A_1A_2\ldots A_n \cap B_1B_2\ldots B_m = \varnothing $. Also if $e\in K$, then for any such $x_1 x_2\ldots x_n$ since $x_i \neq e$ for any $i$, we can choose $A_i$'s such that $e\in \hspace{-0.11in} / \ A_i$  for any $i$. Then again using the fact that each $K_\alpha$ is Hausdorff, we can choose a neighborhood $V$ around $e$  such that $V\cap A_i = \varnothing$ for all $i$ and hence $V \cap A_1A_2\ldots A_n = \varnothing$. Thus we can easily see that the topology $\tau$ on $K$ is Hausdorff.

	
	\paragraph{\textbf{$\tau$ is locally compact:}}
	
	\hspace{0.03in}
	
	\noi Next pick any element $x_1 x_2\ldots x_n \in K$ where $x_i \in K_{\alpha_i}$ for $1\leq i\leq n$. Since each $K_{\alpha_i}$  is locally compact and Hausdorff, we can find a compact neighborhood $C_i$ of $x_i$ in $K_{\alpha_i}$ such that $e \in \hspace{-0.11in} / \ C_i$. Let $U_i$ be an open set in $K_{\alpha_i}$ such that $C_i \subset U_i$ and $e \in \hspace{-0.11in} / \ U_i$. Now consider the map $\phi : U_1 \times U_2 \times \ldots \times U_n \ra U_1U_2\ldots U_n$ given by
	$$ \phi((a_1, a_2, \ldots, a_n)) := (a_1a_2\ldots a_n) $$
	for any  $(a_1, a_2, \ldots, a_n) \in U_1 \times U_2 \times \ldots \times U_n \ $.

	\vspace{0.05in}

	\noi Note that since $e \in \hspace{-0.11in} / \ U_i$ for any $i$, the map $\phi$ is bijective. Moreover, for any basic open set $V_1V_2\ldots V_n \subset U_1U_1\ldots U_n$ we have $\phi^{-1}(V_1V_2\ldots V_n) = V_1\times V_2\times \ldots \times V_n$ which is open in $U_1 \times U_2 \times \ldots \times U_n$ in the product topology. Hence $\phi$ is continuous, granting us that $C_1C_2\ldots C_n = \phi(C_1\times C_2 \times \ldots \times C_n)$ is a compact neighborhood of $x_1x_2\ldots x_n$ in $(K, \tau)$. Also if $e\in K$ then there exists some $\alpha_0$ such that $e\in K_{\alpha_0}$. Since $K_{\alpha_0}$ is locally compact we can always find a compact neighborhood for $e$. Thus we see that $(K, \tau)$ is a locally compact Hausdorff space.

	
	\paragraph{\textbf{Continuity of support with respect to Michael Topology:}}
	
	\hspace{0.03in}
	
	\noi Now in order to see if the map $(x, y) \mapsto supp (p_x*p_y) : K\times K \ra \mathfrak{C}(K)$ is continuous, first recall \cite{JE} that the map $\pi: X \ra M(X)$ given by $x \mapsto p_x$ is a homemorphism onto its image for any locally compact Hausdorff space $X$. Now let $\{(x_\alpha, y_\alpha)\}$ be a net in $K\times K$ that converges to some $(x, y)$ where $x= x_1x_2\ldots x_n$, $y=y_1y_2\ldots y_m \in K$ such that $x_i \in K_{\gamma_i}$ and $y_j \in K_{\delta_j}$ for all $i, j$. Then $\{x_\alpha\}$ and $\{y_\alpha\}$ converges to $x$ and $y$ respectively. Then for each $\alpha$ we have that $x_\alpha = x_1^\alpha x_2^\alpha \ldots x_{n_\alpha}^\alpha$ for some $n_\alpha \in \mathbb{N}$ where $x_i^\alpha \in K_{\alpha_i}$ for $1 \leq i \leq n_\alpha$. But $x_\alpha$ eventually lies in any basic open set of the form $U_1U_2\ldots U_n$ around $x$ where $U_i$ is an open neighborhood of $x_i$ in $K_{\gamma_i}$ that does not contain $e$. Hence without loss of generality for each $\alpha$ we can assume that $n_\alpha =n$ and ${\alpha_i} =  \gamma_i$ for $1\leq i \leq n$. Similarly we can assume that $y_\alpha = y_1^\alpha y_2^\alpha \ldots y_m^\alpha$ where $y_j^\alpha \in K_{\delta_j}$ for $1\leq j \leq m$.
	
	\vspace{0.05in}
	
	\noi Now consider the Michael topology on $\mathfrak{C}(K)$ and recall \cite{NA} that the map 
	$$x\mapsto \{x\} : K \ra \mathfrak{C}(K)$$
	is a homeomorphism onto its image which is also closed in $\mathfrak{C}(K)$. If $\gamma_n \neq \delta_1$ then using the same technique as above, we can easily see that the net $\{(x_1^\alpha \ldots x_n^\alpha y_1^\alpha\ldots\\ \ldots y_m^\alpha)\}_\alpha$ converges to $(x_1\ldots x_ny_1\ldots y_m)$ in $K$. Hence using the homeomorphism mentioned above we have that $supp (p_{x_\alpha}*p_{y_\alpha}) = \{(x_1^\alpha \ldots x_n^\alpha y_1^\alpha\ldots y_m^\alpha)\}$ converges to $\{(x_1\ldots x_ny_1\ldots y_m)\} = supp (p_x*p_y)$ in $\mathfrak{C}(K)$ in Michael Topology.
	
	\vspace{0.05in}
	
	\noi Now let $\gamma_n = \delta_1$. For any $i_0 \in \{1, 2, \ldots n\}$, pick and fix open neighborhoods $U_i$ around $x_i$ in $K_{\gamma_i}$ for all $i\neq i_0$. Then for any open neighborhood $V$ around $x_{i_0}$ in $K_{\gamma_{i_0}}$ we have that $x_1^\alpha x_2^\alpha \ldots x_{n}^\alpha$ lie eventually in $U_1U_2\ldots U_{i_0-1}VU_{i_0+1}\ldots U_n$. Hence we must have that $\{x_{i_0}^\alpha\}$ lie eventually in $V$, \textit{i.e,} $\{x_{i_0}^\alpha\}$ converges to $x_{i_0}$ in $K_{\gamma_{i_0}}$. Thus  $\{x_i^\alpha\}$ converges to $x_i$ in $K_{\gamma_i}$ and $\{y_j^\alpha\}$ converges to $y_j$ in $K_{\delta_j}$ for each $i, j$. In particular, we have that the net $\{(x_n^\alpha, y_1^\alpha)\}_\alpha$ converges to $(x_n, y_1)$ in $K_{\gamma_n} \times K_{\gamma_n}$.  Since the map $(x, y) \mapsto supp (p_x*p_y) : K_{\gamma_n}\times K_{\gamma_n} \ra \mathfrak{C}(K_{\gamma_n})$ is continuous, we have that $supp(p_{x_n^\alpha}* p_{y_1^\alpha})$ converges to $supp(p_{x_n}*p_{y_1})$ in $\mathfrak{C}(K_{\gamma_n})$. 
	
	\vspace{0.05in}
	
	\noi As discussed before, we also have that the nets $\big{\{}\{x_i^\alpha\}\big{\}}_\alpha$ and $\big{\{}\{y_j^\alpha\}\big{\}}_\alpha$ converge to $\{x_i\}$ and $\{y_j\}$ in $\mathfrak{C}(K_{\gamma_i})$ and $\mathfrak{C}(K_{\delta_j})$  respectively for $1\leq i < n$ and $1<j \leq m$. Hence it follows from the construction of basic open sets of $K$ that $supp (p_{x_\alpha}*p_{y_\alpha}) = \{(x_1^\alpha\ldots x_{n-1}^\alpha)\} supp (p_{x_n^\alpha}* p_{y_1^\alpha}) \{(y_2^\alpha\ldots y_m^\alpha)\}$ converges to $\{(x_1\ldots x_{n-1})\} supp (p_{x_n}* p_{y_1}) \{(y_2\ldots y_m)\} = supp (p_x *p_y)$ in $\mathfrak{C}(K)$.
	
	\vspace{0.05in}
	
	\noi Finally, it follows immediately from the construction of convolution products that the map $*: M(K) \times M(K) \ra M(K)$ is positive bilinear. Using the same technique as above we can see that the restricted map $*|_{M^+(K) \times M^+(K)} : M^+(K) \times M^+(K) \ra M^+(K)$ is continuous. Thus we have that the pair $(K, *)$ indeed forms a semihypergroup.
	
	
	\paragraph{\textbf{$K_\alpha$ as a sub-semihypergroup:}}
	
	\hspace{0.03in}
	
	\noi Note that for each $\alpha$ the map $i_\alpha: K_\alpha \ra K$ given by 
	$$i_\alpha(x) = \begin{cases}
	(x) , & \mbox{  if } x\neq e  \\
	e  , & \mbox{  if } x = e  
	\end{cases}$$  where $(x)$ is a word of length $1$ in $K$, enables us to view $K_\alpha$ as a sub-semihypergroup of $K$ since  $i_\alpha(K_\alpha)*i_\alpha(K_\alpha) \subset i_\alpha(K_\alpha)$ and $i_\alpha$ is a homeomorphism onto its image. 
	
	\vspace{0.05in}
	
	\paragraph{\textbf{Universal property:}}
	
	\hspace{0.03in}

	\noi Now let $H$ be a semihypergroup and $\phi_\alpha : K_\alpha \ra H$ be a homomorphism for each $\alpha$. For each $x=x_1x_2\ldots x_n \in K$ where $x_i \in K_{\alpha_i}$ for $1\leq i \leq n$ define
	$$\Gamma(p_x) := p_{\phi_{\alpha_1}(x_1)} * p_{\phi_{\alpha_2}(x_2)} * \ldots * p_{\phi_{\alpha_n}(x_n)}  . $$
	\noi We can then extend $\Gamma$ to $M(K)$ as in the proof of Theorem \ref{homtopc} to get a positive continuous linear map $\Gamma: M(K) \ra M(H)$.
	
	\vspace{0.05in}
	
	\noi Since for any $x\in K_\alpha$ we have that $\Gamma(p_(x)) = p_{\phi_\alpha(x)}$, it follows immediately from Theorem \ref{homtopc} that $\Gamma|_{M(K_\alpha)}= \Gamma \circ \Gamma_{i_\alpha}= \Gamma_{\phi_\alpha}$ for each $\alpha$. Now to show that $\Gamma$ is a homomorphism, pick any two words $x=x_1x_2\ldots x_n$, $y=y_1y_2\ldots y_m$ in $K$ where $x_i \in K_{\alpha_i}$, $y_j \in K_{\beta_j}$ for $1\leq i \leq n$, $1\leq j \leq m$. Note that if $\alpha_n \neq \beta_1$ then we have that

	\begin{eqnarray*}
		\Gamma(p_x*p_y) &=& \Gamma(p_{x_1\ldots x_ny_1\ldots y_m})\\
		&=& p_{\phi_{\alpha_1}(x_1)}* \ldots *p_{\phi_{\alpha_n}(x_n)}*p_{\phi_{\beta_1}(y_1)}*\ldots * p_{\phi_{\beta_m}(y_m)}\\
		&=& \big{(}p_{\phi_{\alpha_1}(x_1)}* \ldots *p_{\phi_{\alpha_n}(x_n)}\big{)}*\big{(}p_{\phi_{\beta_1}(y_1)}*\ldots * p_{\phi_{\beta_m}(y_m)}\big{)}\\
		&=& \Gamma(p_{x_1x_2\ldots x_n}) * \Gamma(p_{y_1y_2\ldots y_m}) = \Gamma(p_x) * \Gamma(p_y)
	\end{eqnarray*}
	as required.
	
	
	\noi Now pick any $x \in K_\gamma$ and $y, z \in K_\beta$ where $\gamma\neq \beta$. Then there exists a net $\{m_\alpha\}$ in $M_F^+(K_\beta)$ such that $m_\alpha$ converges to $(p_y*p_z)$ in $M^+(K_\beta)$. Hence $(p_x*m_\alpha)$ converges to $(p_x*(p_y*p_z))= (p_x*p_y*p_z)$ in $M^+(K)$. Then for each $\alpha$ we have that $m_\alpha = \sum_{i=1}^{n_\alpha} c_i^\alpha p_{x_i^\alpha}$ for some $n_\alpha \in \mathbb{N}$, $c_i^\alpha \in \mathbb{C}$ and $x_i^\alpha \in K_\beta$. Since $\Gamma$ is continuous on $M^+(K)$, we have that

	\begin{eqnarray*}
		\Gamma(p_x*p_y*p_z) &=& \Gamma\big{(}\lim_\alpha (p_x *m_\alpha)\big{)}\\
		&=& \lim_\alpha \Gamma(p_x * m_\alpha)\\
		&=& \lim_\alpha \Gamma\Big{(}p_x * \sum_{i=1}^{n_\alpha} c_i^\alpha p_{x_i^\alpha}\Big{)}\\
		&=&\lim_\alpha \Gamma\Big{(} \sum_{i=1}^{n_\alpha} c_i^\alpha (p_x*p_{x_i^\alpha})\Big{)}\\
		&=& \lim_\alpha  \sum_{i=1}^{n_\alpha} c_i^\alpha \ \Gamma\big{(}p_x*p_{x_i^\alpha}\big{)}\\
		&=& \lim_\alpha \sum_{i=1}^{n_\alpha}c_i^\alpha \Big{(}\Gamma(p_x)*\Gamma(p_{x_i^\alpha})\Big{)}\\
		&=& \lim_\alpha  \Gamma(p_x)*\Big{(}\sum_{i=1}^{n_\alpha}c_i^\alpha\Gamma(p_{x_i^\alpha})\Big{)}\\
		&=&\lim_\alpha \Gamma(p_x) * \Gamma\Big{(} \sum_{i=1}^{n_\alpha}c_i^\alpha p_{x_i^\alpha}\Big{)}\\
		&=& \Gamma(p_x) * \lim_\alpha \Gamma(m_\alpha)
		= \Gamma(p_x) * \Gamma(p_y*p_z)
	\end{eqnarray*}
	\noi where the fourth equality follows since the convolution $(\mu, \nu) \mapsto \mu *\nu : M(K)\times M(K) \ra M(K)$ is bilinear, the fifth and eighth equality follows from the linearity of $\Gamma$ on $M(K)$ and finally the seventh equality follows since the map $(\mu, \nu) \mapsto \mu *\nu : M(H)\times M(H) \ra M(H)$ is bilinear.
	
	
	\noindent Similarly for any $x, y \in K_\gamma$ and $z\in K_\beta$ where $\gamma\neq\beta$ we have that
	$$\Gamma(p_x*p_y*p_z) = \Gamma(p_x*p_y) * \Gamma(p_z)  .$$
	Using these two equalities simultaneously for any two words $x=x_1x_2\ldots x_n$, $y=y_1y_2\ldots y_m$ in $K$ where $x_i \in K_{\alpha_i}$, $y_j \in K_{\beta_j}$ for $1\leq i \leq n$, $1\leq j \leq m$ and $\alpha_n = \beta_1$, we have that
	\begin{eqnarray*}
		\Gamma(p_x*p_y) &=& \Gamma (p_{x_1x_2\ldots x_{n-1}} *p_{x_n}*p_{y_1}*p_{y_2y_3\ldots y_m})\\
		&=& \Gamma(p_{x_1x_2\ldots x_{n-1}}) *\Gamma(p_{x_n}*p_{y_1})*\Gamma(p_{y_2y_3\ldots y_m})\\
		&=& \Gamma(p_{x_1x_2\ldots x_{n-1}}) *\Gamma_{\phi_{\alpha_n}}(p_{x_n}*p_{y_1})*\Gamma(p_{y_2y_3\ldots y_m})\\
		&=& \Gamma(p_{x_1\ldots x_{n-1}}) *\Big{(}\Gamma_{\phi_{\alpha_n}}(p_{x_n})*\Gamma_{\phi_{\alpha_n}}(p_{y_1})\Big{)}*\Gamma(p_{y_2\ldots y_m})\\
		&=& \Gamma(p_{x_1\ldots x_{n-1}})*\Big{(}p_{\phi_{\alpha_n}(x_n)} * p_{\phi_{\alpha_n}(y_1)}\Big{)}*\Gamma(p_{y_2\ldots y_m})\\
		&=& \Big{(}p_{\phi_{\alpha_1}(x_1)} * \ldots * p_{\phi_{\alpha_{n-1}}(x_{n-1})}\Big{)}* \Big{(}p_{\phi_{\alpha_n}(x_n)} *p_{\phi_{\beta_1}(y_1)}\Big{)}\\
		&& * \Big{(}p_{\phi_{\beta_2}(y_2)} * \ldots * p_{\phi_{\beta_m}(y_m)}\Big{)}\\
		&=& \Big{(}p_{\phi_{\alpha_1}(x_1)} * \ldots * p_{\phi_{\alpha_{n-1}}(x_{n-1})}* p_{\phi_{\alpha_n}(x_n)}\Big{)}\\
		&& *\Big{(}p_{\phi_{\beta_1}(y_1)}* p_{\phi_{\beta_2}(y_2)} * \ldots * p_{\phi_{\beta_m}(y_m)}\Big{)}\\
		&=& \Gamma\big{(}p_{x_1x_2\ldots x_n}\big{)} * \Gamma\big{(}p_{y_1y_2\ldots y_m}\big{)} \ = \ \Gamma(p_x)*\Gamma(p_y)\end{eqnarray*}
	\noi as required. Note that here the fourth equality follows since $\Gamma_{\phi_{\alpha_n}}: M(K_{\alpha_n}) \ra M(H)$ induced by $\phi_{\alpha_n}$ is a homomorphism and the fifth equality follows from the construction of $\Gamma_{\phi_{\alpha_n}}$ as shown in Theorem \ref{homtopc}.
	
	\vspace{0.05in}
	
	
	\noindent Finally, note that if $\Theta: M(K) \ra M(H)$ is another positive linear continuous homomorphism such that $\Theta|_{M(K_\alpha)} = \Gamma_{\phi_\alpha}$ then for any $x_1x_2\ldots x_n \in K$ where $x_i\in K_{\alpha_i}$ for $1\leq i \leq n$ we have that
	\begin{eqnarray*}
		\Theta(p_{x_1x_2\ldots x_n})&=& \Theta (p_{x_1}*p_{x_2} * \ldots * p_{x_n})\\
		&=& \Theta(p_{x_1})*\Theta(p_{x_2})* \ldots *\Theta(p_{x_n})\\
		&=& \Theta|_{M(K_{\alpha_1})}(p_{x_1})*\Theta|_{M(K_{\alpha_2})}(p_{x_2})*\ldots * \Theta|_{M(K_{\alpha_n})}(p_{x_n})\\
		&=& \Gamma_{\phi_{\alpha_1}}(p_{x_1}) * \Gamma_{\phi_{\alpha_2}}(p_{x_2})* \ldots * \Gamma_{\phi_{\alpha_n}}(p_{x_n})\\
		&=& \Gamma|_{M(K_{\alpha_1})}(p_{x_1})*\Gamma|_{M(K_{\alpha_2})}(p_{x_2})*\ldots * \Gamma|_{M(K_{\alpha_n})}(p_{x_n})\\
		&=& \Gamma (p_{x_1}) *\Gamma (p_{x_2})*\ldots * \Gamma (p_{x_n})\\
		&=& \Gamma (p_{x_1}*p_{x_2} * \ldots * p_{x_n})= \Gamma (p_{x_1x_2\ldots x_n}) .
	\end{eqnarray*}
	
	
	\noindent Thus we see that the map $\Gamma$ constructed above is unique and hence $(K, *)$ satisfies all the conditions of Definition \ref{deffp} giving us that $K= \prod^*_\alpha K_\alpha \ $.
	\qed
\end{proof}


\begin{remark}
	Note that in classical literature, this natural construction fails to provide a free product of topological semigroups or groups. But in the case of pure semihypergroups, the natural construction suffices to provide a free product structure, \textit{i.e}, in the most trivial case, when applied to a family of semigroups which are pure semihypergroups, the natural construction explained above results in a semihypergroup, rather than a semigroup.
\end{remark}


\section{Open Questions and Further Work}
\label{Open Problems}

\noi As discussed briefly in the introduction, the lack of prior extensive research since its inception and the significant examples available in coset and orbit spaces of locally compact groups, Lie groups and homogeneous spaces, makes way to a number of exciting new avenues of research on semihypergroups. 

\vspace{0.03in}

\noi We have already discussed \cite{CB} spaces of almost periodic and weakly almost periodic functions, amenability, ideals and homomorphisms on a semihypergroup. The following is some immediate potential problems related to this article that we are currently working on and/or intend to work on in near future.

\vspace{0.1 in}

\noi \textbf{Problem 1:} Investigate the behavior of free products on semihypergroups for compact case, thereby exploring amenability properties of the resulting free product. While the compactness of the individual components make some of the deductions easier enabling us to drop some prior restrictions, as in most cases, the lack of an algebra in the underlying space prevents any direct extension of similar results \cite{MO,MK} to this setup.

\vspace{0.1in}

\noi \textbf{Problem 2:} Generalize the construction of the above free product for amalgamated products as well as investigate the properties of the resulting semihypergroup, similar to the works in \cite{MK}. 

\vspace{0.1in}

\noi \textbf{Problem 3:} Amend and update the construction of the above free product for a family of hypergroups, such that the resulting free product becomes a hypergroup. This needs updating the construction of the natural topology $\tau$ so that it takes into account the involution of respective hypergroups, eventually resulting in a natural involution of the form $x_1x_2\ldots x_{n-1}x_n \mapsto x_{n}^- x_{n-1}^-\ldots x_2^-x_1^- $.



\section{acknowledgements}
	
	\noi The author would like to sincerely thank her doctoral thesis advisor  Dr. Anthony To-Ming Lau, for suggesting the topic of this study and for the helpful discussions during the course of this work. She is also grateful to the referee for the valuable comments and suggestions.

%


%


\begin{thebibliography}{}
	
	\bibitem{CB} {Bandyopadhyay, C. : \textit{Analysis on Semihypergroups: Function Spaces, Homomorphisms and Ideals}. Semigroup Forum. 100, 671-697 (2020) }
	
	\bibitem{MI}{Berglund, J. F., Junghenn, H. D., Milnes, P. : \textit{Analysis on Semigroups}. Canadian Mathematical Society Series of Monographs and Advanced Texts, Wiley Interscience Publication. 24 (1989)}
	
	\bibitem{DA}{Day, M. M. : \textit{Amenable Semigroups}. Illinois J. Math. 1, 509-544 (1957)}
	
	\bibitem{DU}{ Dunkl, C. F. : \textit{The Measure Algebra of a Locally Compact Hypergroup}. Trans. Amer. Math. Soc. 179, 331-348 (1973)}
	
	\bibitem{FO2}{Forrest, B. : \textit{Amenability and Ideals in $A(G)$}.  J. Austral. Math. Soc. Ser. A 53, no. 2,  143-155 (1992)}
	
	\bibitem{FO3} {Forrest, B., Spronk, N., Samei, E. :  \textit{Convolutions on Compact Groups and Fourier Algebras of Coset Spaces}.  Studia Math. 196 no. 3,  223-249 (2010) }
	
	\bibitem{GH}{Ghahramani, F.,  Medgalchi, A. R. : \textit{Compact Multipliers on Weighted Hypergroup Algebras}. Math. Proc. Chembridge Philos. Soc. 98,   493-500 (1985)}
	
	\bibitem{GH1}{Ghahramani, F., Lau, A. T., Losert, V. : \textit{Isometric Isomorphisms between Banach Algebras Related to Locally Compact Groups}.  Trans. Amer. Math. Soc. 321 no. 1, 273-283 (1990)}
	
	\bibitem{JE}{Jewett, R. I. : \textit{Spaces with an Abstract Convolution of Measures}. Advances in Mathematics. 18, no. 1,  1-101 (1975)}
	
	\bibitem{LA1}{ Lau, A. T. : \textit{Uniformly Continuous Functions on Banach Algebras}. Colloq. Math. 51,   195-205 (1987)}
	
	\bibitem{LDS} {Lau, A. T., Dales, H. G., Strauss, D. P. : \textit{Banach Algebras on Semigroups and on their Compactifications}. Mem. Amer. Math. Soc. 205, no. 966, vi+165 pp (2010)}
	
	\bibitem{MT} {Michael, E. : \textit{Topologies on Spaces of Subsets}. Trans. Amer. Math. Soc. 71,  152-182 (1951)}
	
	\bibitem{MIT} {Mitchell, T. :  \textit{Topological semigroups and fixed points}. Illinois J. Math. 14, 630–641 (1970)}
	
	\bibitem{MO} {Morris, S. A. :  \textit{Free Products of Topological Groups}. Bull. Austral. Math. Soc.  4, 17-29 (1971)}
	
	\bibitem{MK} {Morris, S. A.,   Khan, M. S. :  \textit{Free products of topological groups with central amalgamation. I}. Trans. Amer. Math. Soc. 273,  405-416 (1982)}
	
	\bibitem{NA}{ Nachbin, L. :  \textit{On the Finite Dimensionality of Every Irreducible Representation of a Compact Group}. Proc. Amer. Math. Soc. 12,   11-12 (1961)}
	
	\bibitem{PI}{Pier, J-P. :  \textit{Amenable Banach algebras}. Pitman Research Notes in Mathematics Series, 172. Longman Scientific and Technical, Harlow; John Wiley and Sons Inc., New York, (1988)}
	
	
	\bibitem{RO}{Ross, K. A. : \textit{Hypergroups and Centers of Measure Algebras}. Symposia Mathematica, Vol. XXII,   189-203 (1977)}
	
	
	\bibitem{SP}{Spector, R. : \textit{Apercu de la theorie des hypergroups, (French) Analyse harmonique sur les groupes de Lie (Sém. Nancy-Strasbourg, 1973-75)}.  Lecture Notes in Math. 497, Springer-Verlag, New York, 643-673 (1975)}
	
	\bibitem{ZE}{Zeuner, H. : \textit{One Dimensional Hypergroups}. Advances in Mathematics 76, no. 1,   1-18 (1989)}
	
	%
\end{thebibliography}


\end{document}